\documentclass[11pt,a4paper]{amsart}
\usepackage{amssymb}
\input xy
\xyoption{all}
\theoremstyle{plain}
\newtheorem{Thm}{Theorem}[section]
\newtheorem{Cor}[Thm]{Corollary}
\newtheorem{Lem}[Thm]{Lemma}
\newtheorem{Prop}[Thm]{Proposition}
\newtheorem{Prob}[Thm]{Problem}
\newtheorem*{thma}{Main Theorem}

\theoremstyle{definition}
\newtheorem{Def}[Thm]{Definition}

\theoremstyle{remark}
\newtheorem{Rem}[Thm]{Remark}

\begin{document}

\title[Holomorphic factorization of mappings into $\mbox{SL}_n(\mathbb{C})$]
{Holomorphic factorization of  \\ mappings into $\mbox{SL}_n(\mathbb{C})$}
\author{Bj\"orn Ivarsson \and Frank Kutzschebauch}
\address{Departement Mathematik\\
Universit\"at Bern\\
Sidlerstrasse 5, CH--3012 Bern, Switzerland}
\email{bjoern.ivarsson@math.unibe.ch}
\email{frank.kutzschebauch@math.unibe.ch}
\thanks{Ivarsson supported by the Wenner-Gren Foundations.
Kutzschebauch partially supported by Schweizerischer Nationalfonds Grant 200021-116165 }
\date{\today}
\bibliographystyle{amsalpha}
\begin{abstract} We solve \textsc{Gromov's} \textsc{Vaserstein} problem. Namely, we show that a null-homotopic holomorphic mapping from a finite dimensional reduced Stein space into $\mbox{SL}_n(\mathbb{C})$ can be factored into a finite product of unipotent matrices with holomorphic entries.
\end{abstract}
\maketitle

\tableofcontents

\section{Introduction}

It is standard material in a Linear Algebra course that the group $\mbox{SL}_m(\mathbb{C})$ is
generated by elementary matrices $E+ \alpha e_{ij} \ i\ne j$, i.e., matrices with 1's on the diagonal and all entries
outside the diagonal are zero, except one entry.  Equivalently every matrix $A \in \mbox{SL}_m(\mathbb{C})$
can be written as a finite product of upper and lower diagonal unipotent matrices (in interchanging order). The same question for matrices in $\mbox{SL}_m(R)$ where $R$ is a commutative ring 
instead of the field $\mathbb{C}$ is much more delicate. For example if $R$ is the ring of
complex valued functions (continuous, smooth, algebraic or holomorphic)  from a space $X$ the problem amounts to find for a given map $f : X \to \mbox{SL}_m(\mathbb{C})$ a factorization as a product of upper and lower diagonal unipotent matrices
\begin{equation*}f(x) = \left(\begin{matrix} 1 & 0 \cr G_1(x) & 1 \cr \end{matrix} \right)   
\left(\begin{matrix} 1 & G_2(x) \cr 0 & 1 \cr \end{matrix} \right)  \ldots \left(\begin{matrix} 1 & G_N(x)\cr 0 & 1 \cr \end{matrix} \right) 
\end{equation*}
where the $G_i$ are maps $G_i : X \to \mathbb{C}^{m(m-1)/2}$.

Since any product of (upper and lower diagonal) unipotent matrices is homotopic to a constant map (multiplying each entry outside the diagonals by $ t \in [0, 1]$ we get a homotopy to the identity matrix), one has to assume that the given map $f : X \to \mbox{SL}_m(\mathbb{C})$ is homotopic to a constant map or as we will say null-homotopic. In particular this assumption holds if the space $X$ is contractible.

This very general problem has been studied in the case of polynomials of $n$ variables. For $n=1$, i.e.,
$f : X \to \mbox{SL}_m(\mathbb{C})$ a polynomial map (the ring $R$ equals $\mathbb{C}[z]$) it is an easy consequence of the fact that $\mathbb{C}[z]$ is an Euclidean ring  that such $f$ factors through a product of upper and lower diagonal unipotent matrices. For $m=n=2$ the following
counterexample was found by \textsc{Cohn} \cite{CohnSGL2R}: the matrix
 $$\left(\begin{matrix} 1-z_1z_2 &  z_1^2\\ -z_2^2 & 1+z_1z_2 \end{matrix}\right)\in \mbox{SL}_2(\mathbb{C}[z_1,z_2]) $$ does not decompose as a finite product of unipotent matrices.

For $m\ge 3$ (and any $n$) it is a deep result of \textsc{Suslin} \cite{SuslinSSLGRP} that any matrix in $\mbox{SL}_m(\mathbb{C}[\mathbb{C}^n])$ decomposes as a finite product of unipotent (and equivalently elementary) matrices. 
More results in the algebraic setting can be found in \cite{SuslinSSLGRP}  and \cite{GrunewaldGSL2}.
For a connection to the Jacobian problem on $\mathbb{C}^2$ see \cite{WrightAFPS}.

In the case of continuous complex valued functions on a topological space $X$ the problem was
studied and partially solved  by \textsc{Thurston} and \textsc{Vaserstein} \cite{ThurstonVasersteinK1ES} and  then finally solved by \textsc{Vaserstein} \cite{VasersteinRMDPDFAO}, see Theorem \ref{Vaserstein}.

It is natural to consider the problem for rings of holomorphic functions on Stein spaces, in particular on
$\mathbb{C}^n$. Explicitly this problem was posed by \textsc{Gromov} in his groundbreaking paper \cite{GromovOPHSEB}  where he extends the classical \textsc{Oka-Grauert} theorem from bundles with homogeneous fibers to fibrations with elliptic fibers, e.g., fibrations admitting a dominating spray (for definition see \ref{definspray}).
In spite of the above mentioned result of \textsc{Vaserstein} he calls it the 

\medskip\noindent
{\bf {Vaserstein problem:}} (see \cite[sec 3.5.G]{GromovOPHSEB}) 

{\sl Does every holomorphic map $\mathbb{C}^n \to \mbox{SL}_m(\mathbb{C})$ decompose into a finite
product of holomorphic maps sending $\mathbb{C}^n$ into unipotent subgroups in $\mbox{SL}_m(\mathbb{C})$?}

\medskip
\textsc{Gromov's} interest in this question comes from the question about s-homotopies (s for spray). In this particular
example the spray  on $ \mbox{SL}_m(\mathbb{C})$ is that coming from the multiplication with
unipotent matrices. Of course one cannot use the upper and lower diagonal unipotent matrices only 
to get a spray (there is no submersivity at the zero section!), there need to be at least one more unipotent subgroup to be used in the multiplication. Therefore the factorization in a product of upper and lower diagonal matrices seems to be a stronger condition than to find a map into the iterated spray, but since all maximal unipotent subgroups in $\mbox{SL}_m(\mathbb{C})$ are conjugated and the upper and lower diagonal matrices generate $\mbox{SL}_m(\mathbb{C})$ these two problems are in fact equivalent. We refer the reader
for more information on the subject to \textsc{Gromov's} above mentioned paper.

The main result of this paper is a complete positive solution of \textsc{Gromov's} \textsc{Vaserstein} problem, namely we prove

\begin{thma}[see Theorem \ref{t:mainthmrestate}]\label{t:mainthm}
Let $X$ be a finite dimensional reduced Stein space and $f\colon X\to \mbox{SL}_m(\mathbb{C})$ be a holomorphic mapping that is null-homotopic. Then there exist a natural number $K$ and holomorphic mappings $G_1,\dots, G_{K}\colon X\to \mathbb{C}^{m(m-1)/2}$ such that $f$ can be written as a product of upper and lower diagonal unipotent matrices
\begin{equation*}f(x) = \left(\begin{matrix} 1 & 0 \cr G_1(x) & 1 \cr \end{matrix} \right)   
\left(\begin{matrix} 1 & G_2(x) \cr 0 & 1 \cr \end{matrix} \right)  \ldots \left(\begin{matrix} 1 & G_K(x)\cr 0 & 1 \cr \end{matrix} \right). 
\end{equation*}
\end{thma}

The method of proof is an application of the \textsc{Oka-Grauert-Gromov}-principle to certain stratified
fibrations. The existence of a topological section for these fibrations we deduce from \textsc{Vaserstein's}
result. 

We need the principle in it's strongest form suggested by \textsc{Gromov}, completely proven
by \textsc{Forstneri\v c} and \textsc{Prezelj} \cite{ForstnericEHSCS}, see Theorem \ref{t:forstneric} and also \textsc{Forstneri\v c} \cite[Theorem 8.3]{ForstnericOPSSFB}. After the \textsc{Gromov-Eliashberg} embedding theorem for Stein manifolds (see \cite{EliashbergESMAS}, \cite{SchurmannESSMD}) this is to our knowledge the second time this holomorphic h-principle has an application which goes beyond the classical results of \textsc{Grauert}, \textsc{Forster} and \textsc{Rammspott} \cite{GrauertAFHR}, \cite{GrauertHFWLG}, \cite{GrauertAHFWKR}, \cite{ForsterTMSR}, \cite{ForsterAEPSM},  \cite{ForsterHISA},  \cite{ForsterOPGG},  \cite{ForsterAME}.

The paper is organized as follows. In section \ref{overview}  we introduce the fibration and give an overview of the proof. In the next section we
explain how  the \textsc{Oka-Grauert-Gromov}-principle is used in the proof. In sections \ref{technical1} and \ref{technical2} we prove all technical details referred to in earlier sections. In the last section we comment on the number of matrices needed in the multiplication.

The results of the present paper have been announced in Comptes Rendus \cite{IK}, communicated by \textsc{Misha Gromov}.
The authors like to thank him for his interest in the subject.  Also the authors  thank \textsc{Franc Forstneri\v{c}}, \textsc{Josip Globevnik}, \textsc{Marco Slapar},  \textsc{Erik L\o w} and \textsc{Erlend Forn\ae ss Wold} for valuable discussions on the subject over the years. Especially we thank \textsc{Franc Forstneri\v{c}} for including an extra section into his paper \cite{ForstnericOPSSFB} to provide us with the exact version of h-principle we need. A special thank goes to \textsc{Wilberd van der Kallen} for bringing the work of \textsc{Vaserstein} to our attention.

\section{Statement of the result and overview of the proof}
\label{overview}
All complex spaces considered in this paper will be assumed reduced and we will not repeat this every time. We call a complex space $X$ finite dimensional if its smooth part $X\setminus X^{\text{sing}}$ has
finite dimension. Note that this does not imply that they have finite embedding dimension.

 We introduce the following notation. Let $n$ and $K$ be natural numbers. If $K$ is odd write $Z_K\in \mathbb{C}^{n(n-1)/2}$ as $$Z_K=(z_{21,K},\dots,z_{ji,K},\dots,z_{n(n-1),K})$$ for $1\le i < j \le n$ and for $K$ even  $$Z_K=(z_{12,K},\dots,z_{ji,K},\dots,z_{(n-1)n,K})$$ for $1\le j < i \le n$. Now define $M_k\colon \mathbb{C}^{n(n-1)/2}\to \mbox{SL}_n(\mathbb{C})$ as $$M_{2l}(Z_{2l})= \left(\begin{matrix} 1 & z_{12,2l} & \dots & z_{1n,2l} \\ 0 & \ddots & \ddots & \vdots  \\ \vdots & \ddots & \ddots & z_{(n-1)n,2l} \\ 0 & \dots & 0 & 1\end{matrix}\right)$$ and $$M_{2l-1}(Z_{2l-1})= \left(\begin{matrix} 1 & 0 & \dots & 0 \\ z_{21,2l-1} & \ddots & \ddots & \vdots  \\ \vdots & \ddots & \ddots & 0 \\ z_{n1,2l-1} & \dots & z_{n(n-1),2l-1} & 1\end{matrix}\right).$$

\begin{Rem}
In the proofs of our results we will study products $$M_1(Z_1)^{-1}\dots M_K(Z_K)^{-1}.$$ This is done for purely technical reasons. Using automorphisms of $\mathbb{C}^{n(n-1)/2}$ sending a unipotent matrix to it's inverse this is equivalent to a product $M_1(X_1)\dots M_K(X_K)$ in new coordinates $(X_1,\dots, X_K)$.
\end{Rem}

As pointed out in the introduction \textsc{Vaserstein} constructed the factorization in the case of continuous mappings, namely he proved

\begin{Thm}[see {\cite[Theorem 4]{VasersteinRMDPDFAO}}] \label{Vaserstein} 
For any natural number $n$ and an integer $d\ge 0$ there is a natural number $K$ such that for any
 finite dimensional normal topological space $X$  of dimension $d$ and a null-homotopic, that is homotopic to the identity, continuous mapping $f\colon X \to \mbox{SL}_n(\mathbb{C})$ it can be written as a finite product of no more than $K$ unipotent matrices.  That is, one can find continuous mappings $F_l\colon X\to \mathbb{C}^{n(n-1)/2}$, $1\le l \le K$  such that $f(x)=M_1(F_1(x))\dots M_K(F_K(x))$. 
 \end{Thm}
 

We recall the statement of the main result of this paper.

\begin{Thm}\label{t:mainthmrestate}
Let $X$ be a finite dimensional reduced Stein space and $f\colon X\to \mbox{SL}_n(\mathbb{C})$ be a holomorphic mapping that is null-homotopic. Then there exist a natural number $K$ and holomorphic mappings $G_1,\dots, G_{K}\colon X\to \mathbb{C}^{n(n-1)/2}$ such that $$f(x)=M_{1}(G_1(x))\dots M_{K}(G_{K}(x)).$$
\end{Thm}

We have the following corollary which in particular solves the \textsc{Gromov's Vaserstein} problem.

\begin{Cor}
Let $X$ be a finite dimensional reduced Stein space that is topologically contractible and $f\colon X\to \mbox{SL}_n(\mathbb{C})$ be a holomorphic mapping. Then there exist a natural number $K$ and holomorphic mappings $G_1,\dots, G_{K}\colon X\to \mathbb{C}^{n(n-1)/2}$ such that $$f(x)=M_{1}(G_1(x))\dots M_{K}(G_{K}(x)).$$
\end{Cor}

By the definition of the Whitehead $K_1$-group of a ring, see  \cite[page 61]{RosenbergAKTA}, this implies. 

\begin{Cor}
Let $X$ be a finite dimensional reduced Stein space that is topologically contractible and denote by ${\mathcal O}(X)$ the ring of holomorphic functions on $X$. Then $SK_1({\mathcal O}(X))$ is trivial and the determinant induces an isomorphism $\det \colon K_1({\mathcal O}(X))\to {\mathcal O}(X)^\star$.
\end{Cor}

The strategy for proving Theorem \ref{t:mainthmrestate} is as follows. Define $\Psi_K\colon (\mathbb{C}^{n(n-1)/2})^K\to \mbox{SL}_n(\mathbb{C})$ as $$\Psi_K(Z_1,\dots,Z_K)=M_1(Z_1)^{-1}\dots M_K(Z_K)^{-1}.$$ We want to show the existence of a holomorphic map $$G=(G_1,\dots, G_K)\colon X\to (\mathbb{C}^{n(n-1)/2})^K$$ such that  $$\xymatrix{ & (\mathbb{C}^{n(n-1)/2})^K \ar[d]^{\Psi_{K}} \\ X \ar[r]_{f} \ar[ur]^{G} & \mbox{SL}_n(\mathbb{C})}$$ is commutative. The result by \textsc{Vaserstein} shows the existence of a continuous map such that the diagram above is commutative. 

We will prove Theorem \ref{t:mainthmrestate} using the \textsc{Oka-Grauert-Gromov} principle for sections of holomorphic submersions over $X$. One candidate submersion would be to use the pull-back of $\Psi_K\colon (\mathbb{C}^{n(n-1)/2})^K\to \mbox{SL}_n(\mathbb{C})$. It turns out that $\Psi_K$ is not a submersion at all points in $(\mathbb{C}^{n(n-1)/2})^K$. It is a surjective holomorphic submersion if one removes a certain subset from $(\mathbb{C}^{n(n-1)/2})^K$. Unfortunately the fibers of this submersion are quite difficult to analyze and we therefore elect to study  $$\xymatrix{ & (\mathbb{C}^{n(n-1)/2})^K \ar[d]^{\pi_n\circ \Psi_{K}} \\ X \ar[r]_{\pi_n\circ f} \ar[ur]^{F} & \mathbb{C}^n\setminus \{ 0 \}}$$ where we define the projection $\pi_n\colon \mbox{SL}_n(\mathbb{C})\to \mathbb{C}^n\setminus \{ 0 \}$ to be the
projection of a matrix to its last row:
\begin{equation*}
\pi_n\left( \left(\begin{matrix} z_{11} & \dots & z_{1n} \\ \vdots & \ddots & \vdots \\ z_{n1} & \dots & z_{nn} \end{matrix}\right) \right)=(z_{n1},\dots , z_{nn}).
\end{equation*} 

However, even the map $\Phi_K=\pi_n\circ \Psi_K\colon (\mathbb{C}^{n(n-1)/2})^K\to \mathbb{C}^n\setminus \{0\}$ is not submersive everywhere. We have the following three results about that map which will be proved in later sections.

\begin{Lem}\label{l:submersive}
The mapping $\Phi_K=\pi_n\circ \Psi_K\colon (\mathbb{C}^{n(n-1)/2})^K\to \mathbb{C}^n\setminus \{0\}$ is a holomorphic submersion exactly at points $Z=(Z_1,\dots, Z_K)\in (\mathbb{C}^{n(n-1)/2})^K\setminus S_K$ where for $K\ge 2$ \begin{equation*} \begin{aligned}
S_K&=\left(\bigcap_{1\le 2j+1 < K}\left \{(Z_1,\dots, Z_K)\in (\mathbb{C}^{n(n-1)/2})^K: z_{n1,2j+1}=\dots=z_{n(n-1),2j+1}=0 \right \}\right) \cap \\ \cap& \left(\bigcap_{1\le 2j < K}\left \{(Z_1,\dots, Z_K)\in (\mathbb{C}^{n(n-1)/2})^K: z_{1n,2j}=\dots=z_{(n-1)n,2j}=0 \right \}\right),
\end{aligned}
\end{equation*}
that is the entries in the last row of each lower triangular matrix and the entries in the last column of each upper triangular matrix are 0, except for the $K$-th matrix where no conditions are imposed.
\end{Lem} 
\begin{Lem}\label{l:surjective}
The mapping $\Phi_K=\pi_n\circ \Psi_K\colon (\mathbb{C}^{n(n-1)/2})^K\setminus S_K\to \mathbb{C}^n\setminus \{0\}$ is surjective when $K\ge 3$.
\end{Lem}

\begin{Prop}\label{p:mainprop}
Let $X$ be a finite dimensional reduced Stein space and $f\colon X \to \mbox{SL}_n(\mathbb{C})$ be a null-homotopic holomorphic map. Assume that there exists a natural number $K$ and a continuous map $F\colon X\to (\mathbb{C}^{n(n-1)/2})^K\setminus S_K$ such that  $$\xymatrix{ & (\mathbb{C}^{n(n-1)/2})^K \setminus S_K\ar[d]^{\pi_n\circ \Psi_{K}} \\ X \ar[r]_{\pi_n\circ f} \ar[ur]^{F} & \mathbb{C}^n\setminus \{ 0 \}}$$ is commutative. Then there exists a holomorphic map $G\colon X\to (\mathbb{C}^{n(n-1)/2})^K\setminus S_K$, homotopic to $F$ via continuous maps $F_t\colon X\to (\mathbb{C}^{n(n-1)/2})^K\setminus S_K$, such that the diagram above is commutative for all $F_t$.
\end{Prop}

\begin{proof}[Proof of Theorem \ref{t:mainthmrestate}]
We use induction on $n$. Note that the result is obvious for $n=1$. Suppose that Proposition \ref{p:mainprop} is valid for $n$ and Theorem \ref{t:mainthmrestate} for $n-1$. Put $\Phi_K = \pi_n\circ \Psi_K$. We can find a continuous map $F\colon X \to (\mathbb{C}^{n(n-1)/2})^K\setminus S_K$ for some natural number $K$ such that $f(x)=\Psi_K(F(x))$. Indeed, since a finite dimensional Stein space is finite dimensional as a topological space
the \textsc{Vaserstein} result (Theorem \ref{Vaserstein}) gives us a map $(F_1,\dots, F_{K'})$ into $\left(\mathbb{C}^{n(n-1)/2}\right)^{K'}$. Abusing notation slightly one sees (use Lemma \ref{l:submersive})  that $F=(F_1,\dots, F_{K'}, (0,\dots,1),(0,\dots,0),(0,\dots,-1))$ gives a map from $X$ into $\left( \mathbb{C}^{n(n-1)/2}\right)^{K'+3}\setminus S_{K'+3}$ and putting $K=K'+3$ we have $f(x)=\Psi_K(F(x))$. It follows that $\Psi_K(F(x))f(x)^{-1}=E_n$. Using Proposition \ref{p:mainprop} we know that $F$ is homotopic to a holomorphic map $G$ such that $$\Phi_K(F(x))=\pi_n(f(x))=\Phi_K(G(x))$$ that is the last rows of the matrices $\Psi_K(F(x))$ and $\Psi_K(G(x))$ are equal. Therefore \begin{equation*}
\Psi_K(G(x))f(x)^{-1} = \left(\begin{matrix} \widetilde{f}^1_1(x) & \dots & \widetilde{f}^{n-1}_1(x) & h^n_1(x)    \\ \vdots & \ddots & \vdots & \vdots \\ \widetilde{f}^1_{n-1}(x) & \dots & \widetilde{f}^{n-1}_{n-1}(x) & h^n_{n-1}(x) \\ 0 & \dots & 0 & 1\end{matrix}\right)
\end{equation*}
where all entries are holomorphic. We see that \begin{equation*}
f_{n-1}(x)=  \left(\begin{matrix} \widetilde{f}^1_1(x) & \dots & \widetilde{f}^{n-1}_1(x) \\ \vdots & \ddots & \vdots \\ \widetilde{f}^1_{n-1}(x) & \dots & \widetilde{f}^{n-1}_{n-1}(x) \end{matrix}\right)
\end{equation*}
defines a holomorphic map $f_{n-1}\colon X \to \mbox{SL}_{n-1}(\mathbb{C})$. The homotopy $$\Psi_K(F_t(x))f(x)^{-1}$$ consists of matrices having last row equal to $(0,\dots,0,1)$ and therefore the $(n-1)\times(n-1)$ left upper corner of the matrices $\Psi_K(F_t(x))f(x)^{-1}$ are in $\mbox{SL}_{n-1}(\mathbb{C})$ for all $t$. Since for $t=0$ it is the identity matrix, the map $f_{n-1}\colon X \to \mbox{SL}_{n-1}(\mathbb{C})$ is null-homotopic. 

We use the induction hypotheses to write $f_{n-1}$ as a product of unipotent matrices with holomorphic entries. That is there exists $\widetilde{K}$, a holomorphic map $\widetilde{G}\colon X \to (\mathbb{C}^{(n-1)(n-2)/2})^{\widetilde{K}}\setminus S_{\widetilde{K}}$ such that \begin{equation*}
\widetilde{f}(x)=M_1(\widetilde{G}_1(x))\dots M_{\widetilde{K}}(\widetilde{G}_{\widetilde{K}}(x)).
\end{equation*}
Hence we have \begin{equation*}
\begin{aligned}
&\left(\begin{matrix} E_{n-1} & \begin{matrix} -h^n_1(x) \\ \vdots \\ -h^n_{n-1}(x) \end{matrix} \\ \begin{matrix} 0 & \dots & 0 \end{matrix} & 1 \end{matrix}\right)
\Psi_K(G(x))f(x)^{-1} = \left(\begin{matrix} \widetilde{f}(x) & \begin{matrix} 0 \\ \vdots \\ 0 \end{matrix} \\ \begin{matrix} 0 & \dots & 0 \end{matrix} & 1 \end{matrix}\right) = \\ 
& = \left(\begin{matrix} M_{1}(\widetilde{G}_{1}(x)) & \begin{matrix} 0 \\ \vdots \\ 0 \end{matrix} \\ \begin{matrix} 0 & \dots & 0 \end{matrix} & 1 \end{matrix}\right)
\dots 
 \left(\begin{matrix} M_{\widetilde{K}}(\widetilde{G}_{\widetilde{K}}(x)) & \begin{matrix} 0 \\ \vdots \\ 0 \end{matrix} \\ \begin{matrix} 0 & \dots & 0 \end{matrix} & 1 \end{matrix}\right) 
\end{aligned}
\end{equation*}
and the result follows by induction.
\end{proof}

In order to complete the proof of the theorem we need to establish Proposition \ref{p:mainprop}, Lemma \ref{l:submersive}, and \ref{l:surjective}.

\section{Stratified sprays}

We will introduce the concept of a spray associated with a holomorphic submersion following \cite{GromovOPHSEB} and \cite{ForstnericOPHSWS}. First we introduce some notation and terminology. Let $h\colon Z \to X$ be a holomorphic submersion of a complex manifold $Z$ onto a complex manifold $X$. For any $x\in X$ the fiber over $x$ of this submersion will be denoted by $Z_x$. At each point $z\in Z$ the tangent space $T_zZ$ contains {\it the vertical tangent space} $VT_zZ=\ker Dh$. For holomorphic vector bundles $p\colon E \to Z$ we denote the zero element in the fiber $E_z$ by $0_z$.

\begin{Def} \label{definspray}
Let $h\colon Z \to X$ be a holomorphic submersion of a complex manifold $Z$ onto a complex manifold $X$. A spray on $Z$ associated with $h$ is a triple $(E,p,s)$, where $p\colon E\to Z$ is a holomorphic vector bundle and $s\colon E\to Z$ is a holomorphic map such that for each $z\in Z$ we have 
\begin{itemize}
\item[(i)]{$s(E_z)\subset Z_{h(z)}$,}
\item[(ii)]{$s(0_z)=z$, and}
\item[(iii)]{the derivative $Ds(0_z)\colon T_{0_z}E\to T_zZ$ maps the subspace $E_z\subset T_{0_z}E$ surjectively onto the vertical tangent space $VT_zZ$.}
\end{itemize} 
\end{Def} 

\begin{Rem}
We will also say that the submersion admits a spray. A spray associated with a holomorphic submersion is sometimes called a (fiber) dominating spray.
\end{Rem}

One way of constructing dominating sprays, as pointed out by \textsc{Gromov}, is to find finitely many $\mathbb{C}$-complete vector fields that are tangent to the fibers and span the tangent space of the fibers at all points in $Z$. One can then use the flows $\varphi_j^t$ of these vector fields $V_j$ to define $s\colon Z\times \mathbb{C}^N\to Z$ via $s(z,t_1,\dots, t_N)=\varphi_1^{t_1}\circ \dots \circ \varphi_N^{t_N}(z)$ which gives a spray. 

\begin{Def}
Let $X$ and $Z$ be complex spaces. A holomorphic map $h\colon Z \to X$ is said to be a submersion if for each point $z_0\in Z$ it is locally equivalent via a fiber preserving biholomorphic map to a projection $p\colon U\times V \to U$, where $U\subset X$ is an open set containing $h(z_0)$ and $V$ is an open set in some $\mathbb{C}^d$.  
\end{Def}

We will need to use stratified sprays which are defined as follows.

\begin{Def} 
\label{definstratifiedspray}
We say that a submersion $h\colon Z\to X$ admits stratified sprays if there is a descending chain of closed complex subspaces $X=X_m\supset \cdots \supset X_0$ such that each stratum $Y_k = X_k\setminus X_{k-1}$ is regular and the restricted submersion $h\colon Z|_{Y_k}\to Y_k$ admits a spray over a small neighborhood of any point $x\in Y_k$.  
\end{Def}
\begin{Rem}
We say that the stratification $X=X_m\supset \cdots \supset X_0$ is associated with the stratified spray.
\end{Rem}

In \cite{ForstnericEHSCS}, see also \cite[Theorem 8.3]{ForstnericOPSSFB}, the following theorem is proved. 
\begin{Thm}\label{t:forstneric}
Let $X$ be a Stein space with a descending chain of closed complex subspaces $X=X_m\supset \cdots \supset X_0$ such that each stratum $Y_k = X_k\setminus X_{k-1}$ is regular. Assume that $h\colon Z \to X$ is a holomorphic submersion which admits stratified sprays then any continuous section $f_0\colon X\to Z$ such that $f_0|_{X_0}$ is holomorphic can be deformed to a holomorphic section $f_1\colon X\to Z$ by a homotopy that is fixed on $X_0$. 
\end{Thm}

\begin{Lem}\label{l:spray}
The holomorphic submersions $\Phi_K\colon (\mathbb{C}^{n(n-1)/2})^K\setminus S_K \to \mathbb{C}^n\setminus \{0\}$, for $K\ge 3$, admit stratified sprays.
\end{Lem}
This lemma will be established in Section \ref{technical2}. Assuming it true for the moment we prove Proposition \ref{p:mainprop}.
\begin{proof}[Proof of Proposition \ref{p:mainprop}]
Assume that $K$ is a natural number so that there exists a continuous map $F\colon X \to (\mathbb{C}^{n(n-1)/2})^K\setminus S_K$ such that  
$$\xymatrix{ & (\mathbb{C}^{n(n-1)/2})^K \setminus S_K\ar[d]^{\Phi_{K}} \\ X \ar[r]_{\pi_n\circ f} \ar[ur]^{F} & \mathbb{C}^n\setminus \{ 0 \}}$$
is commutative. Put $Y= (\mathbb{C}^{n(n-1)/2})^K\setminus S_K$ and $p=\pi_n\circ f$. Define the pull-back of $\left( Y, \Phi_K, \mathbb{C}^n\setminus \{0\} \right)$ via $p \colon X\to \mathbb{C}^n\setminus \{0\}$ as $(p^{\star}Y , p^{\star}\Phi_K, X)$ where $$ p^{\star}Y=\left \{(x,Z)\in X\times Y; p(x)= \Phi_K (Z) \right \}  $$ and $p^{\star}\Phi_K(x,Z)=x$. Using that $\Phi_K$ is a holomorphic submersion we see that $p^{\star} \Phi_K$ is a holomorphic submersion. The continuous mapping $F$ defines a continuous section $$p^{\star}F(x)=(x,F(x))$$ of $(p^{\star}Y , p^{\star}\Phi_K, X)$. We need to show that $(p^{\star}Y , p^{\star}\Phi_K, X)$ admits stratified sprays. Let $\mathbb{C}^n\setminus \{ 0 \} = V_m \supset \cdots \supset V_0$ be the stratification of $\mathbb{C}^n\setminus \{ 0 \}$ corresponding to the stratified spray of $(Y, \Phi_K, \mathbb{C}^n\setminus \{ 0 \})$. Define $X_j=p^{-1}(V_j)$ for $0\le j \le m$. These complex subspaces need to be stratified in order for us to apply Theorem \ref{t:forstneric}. Define $X_{0,i}=X_{0,i-1}^{\text{sing}}$ for $i\ge 1$ and $X_{0,0}=X_0$. This defines a stratification of $X_0$ since $X_{0,J}=\emptyset$ when $J > L$ for some $L$ since the singularity set of reduced complex spaces has strictly lower dimension than the space itself. We continue by putting $X_{j,0}=X_j$ and $X_{j,i}=X_{j,i-1}^{\text{sing}}\cup X_{j-1}$ for $i\ge 1$. Since X is finite dimensional this gives a stratification of $X$. Now the result follows by Theorem \ref{t:forstneric}. 
\end{proof}

\section{Proof of Lemma \ref{l:surjective} and \ref{l:submersive}}
\label{technical1}

Recall that \begin{equation*} \begin{aligned}
S_K&=\left(\bigcap_{1\le 2j+1 < K}\left \{(Z_1,\dots, Z_K)\in (\mathbb{C}^{n(n-1)/2})^K: z_{n1,2j+1}=\dots=z_{n(n-1),2j+1}=0 \right \}\right) \cap \\ \cap& \left(\bigcap_{1\le 2j < K}\left \{(Z_1,\dots, Z_K)\in (\mathbb{C}^{n(n-1)/2})^K: z_{1n,2j}=\dots=z_{(n-1)n,2j}=0 \right \}\right).
\end{aligned}
\end{equation*}
We begin by proving Lemma \ref{l:surjective}.
\begin{proof}[Proof of Lemma \ref{l:surjective}]
First note that the set $S_K$ is invariant under the automorphism in $(\mathbb{C}^{n(n-1)/2})^K$ replacing $M_1(Z_1)^{-1}\dots M_K(Z_K)^{-1}$ with $M_1(X_1)\dots M_K(X_K)$. Also note that if $$\pi_n\left(M_1(X_1)\dots M_K(X_K)\right)$$ is surjective then $$\pi_n\left(M_1(X_1)\dots M_{K+J}(X_{K+J})\right)$$ is surjective when $J\ge 0$, since $S_{K+J}\subset \{(X_1,\dots,X_{K+J});(X_1,\dots, X_K)\in S_K\}$. Therefore is enough to show the lemma for 
$$\pi_n\left(M_1(X_1)M_2(X_2)M_3(X_3)\right).$$ First 
\begin{equation*}
\begin{aligned}
&\pi_n \left(M_1(X_1)M_2(X_2)\right)= \pi_n \left( \left(\begin{matrix} 1 & 0 & \dots & 0 \\ x_{21,1} & \ddots & \ddots & \vdots \\ \vdots & \ddots & \ddots & 0 \\ x_{n1,1} & \dots & x_{n(n-1),1} & 1 \end{matrix}\right) 
\left(\begin{matrix} 1 & x_{12,2} & \dots & x_{1n,2} \\ 0 & \ddots & \ddots & \vdots \\ \vdots & \ddots & \ddots & x_{(n-1)n,2} \\ 0 & \dots & 0 & 1 \end{matrix}\right)
\right) \\ &=\left( x_{n1,1}\; , \  x_{n2,1}+x_{n1,1}x_{12,2}\; , \  \dots \; ,\  x_{n(n-1),1}+\sum_{j=1}^{n-2} x_{nj,1}x_{j(n-1),2} \; ,  1+ \sum_{j=1}^{n-1} x_{nj,1}x_{jn,2}  \right).
\end{aligned}
\end{equation*}
It is clear that we can map onto the set $$ \left \{ (a_1, \dots , a_n) \in \mathbb{C}^n\setminus \{ 0 \}; a_1\neq 0  \right \}.$$ To map onto $\mathbb{C}^n\setminus \{ 0 \}$ we need to use a third matrix. Consider matrices $M_1(X_1)$ and $M_2(X_2)$ such that $$\pi_n(M_1(X_1)M_2(X_2)) = (1,a_2,\dots ,a_n).$$ For such matrices we have \begin{equation*}
\begin{aligned}
\pi_n & \left(M_1(X_1)M_2(X_2)M_3(X_3)\right)= \\ & =
\pi_n \left( 
\left(\begin{matrix} \star & \dots & \dots & \star \\ \vdots & \ddots & \ddots & \vdots \\ \star & \dots & \dots & \star \\ 1 & a_2 & \dots & a_n \end{matrix}\right)
\left(\begin{matrix} 1 & 0 & \dots & 0 \\ x_{21,3} & \ddots & \ddots & \vdots \\ \vdots & \ddots & \ddots & 0 \\ x_{n1,3} & \dots & x_{n(n-1),3} & 1 \end{matrix}\right) 
\right)= \\ &=\left( 1 + \sum_{j=2}^n a_jx_{j1,3}\; ,\  a_2 + \sum_{j=3}^n a_jx_{j2,3}\; ,\  \dots \; ,\  a_n \right).
\end{aligned}
\end{equation*}
We can choose $X_3, a_2, \dots, a_n$ freely to produce any vector in $\mathbb{C}^n\setminus \{ 0 \}$. Note that we cannot produce 0 since this would force $a_2= \dots = a_n=0$.
\end{proof}
We turn to the proof of Lemma \ref{l:submersive}.
\begin{proof}[Proof of Lemma \ref{l:submersive}]
We begin with the base case $K=2$. Let \begin{equation*}
\left( P_{1,1}(Z_1),\dots, P_{n,1}(Z_1) \right)=\pi_n\left( \left(\begin{matrix} 1 & 0 & \dots & 0 \\ z_{21,1} & \ddots & \ddots & \vdots \\ \vdots & \ddots & \ddots & 0 \\ z_{n1,1} & \dots & z_{n(n-1),1} & 1  \end{matrix}\right)^{-1} \right)
\end{equation*}
and 
\begin{equation*}
\begin{aligned}
& \left( P_{1,2}(Z_1,Z_2),\dots, P_{n,2}(Z_1,Z_2) \right)= \\ &=\pi_n\left( \left(\begin{matrix} 1 & 0 & \dots & 0 \\ z_{21,1} & \ddots & \ddots & \vdots \\ \vdots & \ddots & \ddots & 0 \\ z_{n1,1} & \dots & z_{n(n-1),1} & 1  \end{matrix}\right)^{-1} \left(\begin{matrix} 1 & z_{12,2} & \dots & z_{1n,2} \\ 0 & \ddots & \ddots & \vdots \\ \vdots & \ddots & \ddots & z_{(n-1)n,2} \\ 0 & \dots & 0 & 1 \end{matrix}\right)^{-1}  \right).
\end{aligned}
\end{equation*}
We get relations by studying \begin{equation*}
\begin{aligned}
& \left(\begin{matrix} \star & \dots & \star  \\  \vdots & \ddots & \vdots  \\  \star & \dots & \star  \\ P_{1,1}(Z_1) & \dots & P_{n,1}(Z_1) \end{matrix}\right)= \\  &=\left(\begin{matrix} \star & \dots & \star  \\  \vdots & \ddots & \vdots  \\  \star & \dots & \star  \\ P_{1,2}(Z_1,Z_2) & \dots & P_{n,2}(Z_1,Z_2) \end{matrix}\right) \left(\begin{matrix} 1 & z_{12,2} & \dots & z_{1n,2} \\ 0 & \ddots & \ddots & \vdots \\ \vdots & \ddots & \ddots & z_{(n-1)n,2} \\ 0 & \dots & 0 & 1 \end{matrix}\right).
\end{aligned}
\end{equation*}
They are \begin{equation}\label{e:relations}
\begin{aligned}
P_{1,2}(Z_1,Z_2)&=P_{1,1}(Z_1) \\
P_{2,2}(Z_1,Z_2)&=P_{2,1}(Z_1)- z_{12,2}P_{1,2}(Z_1,Z_2) \\
& \ \, \vdots \\
P_{k,2}(Z_1,Z_2)&=P_{k,1}(Z_1) - \sum_{j=1}^{k-1} z_{jk,2}P_{j,2}(Z_1,Z_2) \\
& \ \, \vdots \\
P_{n,2}(Z_1,Z_2)&=P_{n,1}(Z_1) - \sum_{j=1}^{n-1} z_{jn,2}P_{j,2}(Z_1,Z_2).
\end{aligned}
\end{equation}
We need to establish at which points $dP_{1,2}\wedge dP_{2,2} \wedge \dots \wedge dP_{n,2}=0$. Note that $P_{n,1}\equiv 1$. Also note that $dP_{1,1}\wedge \dots \wedge dP_{n-1,1}$ never vanishes  since $dz_{21,1}\wedge \dots \wedge dz_{n1,1}\wedge \dots \wedge dz_{n(n-1),1}$ never vanishes and $P_{1,1},\dots,  P_{n-1,1}$ are components of an automorphism of $\mathbb{C}^{n(n-1)/2}$ induced by $M_1(Z_1)\mapsto M_1(Z_1)^{-1}$. Set $
\Omega_2=dP_{1,2}\wedge \dots \wedge dP_{n,2}$. We will show that this form is zero if and only if $P_{1,2}(Z_1,Z_2)=\dots =P_{n-1,2}(Z_1,Z_2)=0$. Indeed, when we plug in 
 $$ dP_{k,2}(Z_1,Z_2)=dP_{k,1}(Z_1) - \sum_{j=1}^{k-1} P_{j,2}(Z_1,Z_2)\, d z_{jk,2} - \sum_{j=1}^{k-1} z_{jk,2}\, dP_{j,2}(Z_1,Z_2) $$ to calculate $\Omega_2$ we see that the terms in the last sum do not contribute to the result. Next one sees that $\Omega_2$ contains summands of the form $$(P_{k,2})^{n-k}\; dP_{1,1}\wedge \dots \wedge dP_{k,1}\wedge dz_{k(k+1),2}\wedge \dots \wedge dz_{kn,2}$$ and these are the only summands containing the wedge $ dz_{k(k+1),2}\wedge \dots \wedge dz_{kn,2}$. This implies that at points where $\Omega_2$ vanishes we have $$P_{1,2}(Z_1,Z_2)=\dots =P_{n-1,2}(Z_1,Z_2)=0$$ and all other summands involves products of these functions as coefficients. Thus $\Omega_2$ vanishes if and only if $P_{1,2}(Z_1,Z_2)=\dots =P_{n-1,2}(Z_1,Z_2)=0$.

From (\ref{e:relations}) we see that this is equivalent to $P_{1,1}(Z_1)=\dots = P_{n-1,1}(Z_1)=0$. Note that $P_{1,1}, \dots, P_{n-1,1}$ are components of the automorphism of $\mathbb{C}^{n(n-1)/2}$ induced by $M_1(Z_1)\mapsto M_1(Z_1)^{-1}$. Since this automorphism fixes $S_2=\{z_{n1,1}=\dots = z_{n(n-1),1}=0\}$ we conclude that $\Phi_2$ is submersive exactly at points outside $S_2$. In order to make our induction step we need some further properties. Note that at points where $P_{1,2}=\dots = P_{n-1,2}=0$, that is in $S_2$, we have $$dP_{1,2}\wedge \dots \wedge dP_{n-1,2}=dP_{1,1}\wedge \dots \wedge dP_{n-1,1}\neq 0.$$ We also have $dP_{n,1}\equiv 0$ since $P_{n,1}(Z_1)\equiv 1$.

We now consider $K$ odd. Our induction assumptions are besides the description of the non-submersivity set $S_{K-1}$ the following:

In case when $K$ odd and $dP_{1,K-1}\wedge \dots \wedge dP_{n,K-1}=0$ then 
\begin{description}
\item[($I_{K-1}$)]{$dP_{n,K-2}=0$,} 
\item[($II_{K-1}$)]{$dP_{1,K-1}\wedge \dots \wedge dP_{n-1,K-1}\neq 0$, and} 
\item[($III_{K-1}$)]{$P_{j,K-1}=0$ for $1\le j \le n-1$.}
\end{description}
We now describe the induction step from $K-1$ to $K$ when $K$ is odd.
Doing similar calculations as for $K=2$ we get the relations 
\begin{equation}\label{relationsodd}
\begin{aligned}
P_{n,K}(Z_1,\dots,Z_K)&=P_{n,K-1}(Z_1,\dots , Z_{K-1}) \\
P_{n-1,K}(Z_1,\dots,Z_K)&=P_{n-1,K-1}(Z_1,\dots,Z_{K-1})- z_{n(n-1),K}P_{n,K}(Z_1,\dots,Z_{K}) \\
& \ \, \vdots \\
P_{k,K}(Z_1,\dots,Z_K)&=P_{k,K-1}(Z_1,\dots,Z_{K-1}) - \sum_{j=k+1}^{n} z_{jk,K}P_{j,K}(Z_1,\dots,Z_K) \\
& \ \, \vdots \\
P_{1,K}(Z_1,\dots,Z_K)&=P_{1,K-1}(Z_1,\dots,Z_{K-1}) - \sum_{j=2}^{n} z_{j1,K}P_{j,K}(Z_1,\dots,Z_K).
\end{aligned}
\end{equation}
We see that 
\begin{equation*}
\begin{aligned}
\Omega_K&=dP_{n,K}\wedge dP_{n-1,K}\wedge \dots \wedge dP_{1,K}=\\ &=dP_{n,K-1}\wedge \dots \wedge dP_{1,K-1} + \text{ terms involving } dz_{jl,K}.
\end{aligned}
\end{equation*}
At points where $\Omega_K$ vanishes the form $\Omega_{K-1}=dP_{n,K-1}\wedge \dots \wedge dP_{1,K-1}$ must vanish since it involves no terms in $dz_{jl,K}$. By the induction hypotheses this forces $$Z\in \{(Z_1,\dots, Z_K); (Z_1,\dots, Z_{K-1}) \in S_{K-1}\}.$$ 
A calculation shows that at these points $P_{n,K-1}(Z_1,\dots ,Z_{K-1})=P_{n,K}(Z_1,\dots ,Z_{K})=1$. Plugging in 
\begin{equation*}
\begin{aligned} 
dP_{k,K}(Z_1,\dots,Z_K)&=dP_{k,K-1}(Z_1,\dots,Z_{K-1}) - \\ & - \sum_{j=k+1}^{n} P_{j,K}(Z_1,\dots,Z_K)\; d z_{jk,K} - \sum_{j=k+1}^{n} z_{jk,K}\, dP_{j,K}(Z_1,\dots,Z_K)\end{aligned}
\end{equation*}
to calculate $\Omega_K$ we see that the terms in the last sum do not contribute to the result. We also find a term of the following form 
\begin{equation*}
(P_{n,K})^{n-1}\, dP_{n,K}\wedge dz_{n(n-1),K}\wedge \dots \wedge dz_{nk,K} \wedge \dots \wedge dz_{n1,K}.
\end{equation*}
Since $P_{n,K}=1$ at these points we see that we must have $dP_{n,K}=dP_{n,K-1}=0$ in order for  $ \Omega_K$ to vanish and this obviously implies that $\Omega_K$ vanishes.  We have 
$$ P_{n,K-1}=P_{n,K-2} - \sum_{j=1}^{n-1} z_{jn,K-1}P_{j,K-1}$$
and 
$$ dP_{n,K-1}=dP_{n,K-2} - \sum_{j=1}^{n-1} P_{j,K-1}\, dz_{jn,K-1}- \sum_{j=1}^{n-1} z_{jn,K-1}\,dP_{j,K-1}. $$
By the induction assumption at points where $\Omega_{K-1}$ vanishes we have $dP_{n,K-2}=0$  and $P_{j,K-1}=0$ for $1\le j \le n-1$. Therefore  $$ dP_{n,K-1}=- \sum_{j=1}^{n-1} z_{jn,K-1}\,dP_{j,K-1}$$ at these points. Moreover by the induction hypotheses $dP_{1,K-1}\wedge \dots \wedge dP_{n-1,K-1}\neq 0$ meaning that $dP_{1,K-1}, \dots , dP_{n-1,K-1}$ are linearly independent at these points,
which implies $z_{jn,K-1}=0$ for $1\le j \le n-1$. Therefore the mapping is non-submersive exactly in $S_K$. 


To pass from $K$ to $K+1$, that is from odd to even, we like to establish $(IV_{K})$ and $(V_{K})$ below. We have already established $(IV_K)$ above. To prove $(V_K)$ look at
 \begin{equation*}
\begin{aligned}
&dP_{1,K}\wedge \dots \wedge dP_{n-1,K}=\\ &=dP_{1,K-1}\wedge \dots \wedge dP_{n-1,K-1} + \text{ terms involving } dz_{jl,K} \neq 0
\end{aligned}
\end{equation*}
The non-vanishing of the left hand side at points in $S_K$ follows from $(II_{K-1})$. This concludes the induction step from $K-1$ to $K$ for odd $K$. However, let us for future use explicitly state that: \begin{equation}\label{inductiveassumption}
dP_{n,K} \text{ vanishes at all points in }S_K.  
\end{equation}  

Consider the case $K$ even. Our induction assumption are besides the description of the non-submersivity set $S_{K-1}$ the following: 

In case when $K$ is even and $dP_{1,K-1}\wedge \dots \wedge dP_{n,K-1}=0$ then 
\begin{description}
\item[($IV_{K-1}$)]{$dP_{1,K-1}\wedge \dots \wedge dP_{n-1,K-1}\neq 0$ and}
\item[($V_{K-1}$)]{$P_{n,K-1}=1$.}
\end{description}

 We have \begin{equation}\label{e:even}
\begin{aligned}
P_{1,K}&=P_{1,K-1} \\
P_{2,K}&=P_{2,K-1}- z_{12,K}P_{1,K}\\
& \ \, \vdots \\
P_{n,K}&=P_{n,K-1} - \sum_{j=1}^{n-1} z_{jn,K}P_{j,K}.
\end{aligned}
\end{equation}

We see that 
\begin{equation*}
\begin{aligned}
\Omega_K&=dP_{1,K}\wedge \dots \wedge dP_{n,K}=\\ &=dP_{1,K-1}\wedge \dots \wedge dP_{n,K-1} + \text{ terms involving } dz_{jl,K}.
\end{aligned}
\end{equation*}
At points where $\Omega_K$ vanishes the form $\Omega_{K-1}=dP_{1,K-1}\wedge \dots \wedge dP_{n,K-1}$ must vanish since it involves no terms in $dz_{jl,K}$. By the induction hypotheses this forces $$Z\in \widetilde{S}_{K-1}=\{(Z_1,\dots, Z_K); (Z_1,\dots, Z_{K-1}) \in S_{K-1}\}.$$ We will show that $\Omega_K$ vanishes at points in $\widetilde{S}_{K-1}$ if and only if $P_{1,K}(Z_1,\dots,Z_K)=\dots =P_{n-1,K}(Z_1,\dots, Z_{K})=0$. Indeed, when we plug in 
 $$ dP_{k,K}=dP_{k,K-1} - \sum_{j=1}^{k-1} P_{j,K}\, d z_{jk,K} - \sum_{j=1}^{k-1} z_{jk,K}\, dP_{j,K} $$ to calculate $\Omega_K$ we see that the terms in the last sum do not contribute to the result. Next one sees that $\Omega_K$ contains summands of the form $$(P_{k,K})^{n-k}\; dP_{1,K-1}\wedge \dots \wedge dP_{k,K-1}\wedge dz_{k(k+1),K}\wedge \dots \wedge dz_{kn,K}$$ and these are the only summands containing the wedge $ dz_{k(k+1),K}\wedge \dots \wedge dz_{kn,K}$. 
Using $(IV_{K-1})$ we see that for all $1\le k \le n-1$ the wedge $ dP_{1,K-1}\wedge \dots \wedge dP_{k,K-1}$ never vanishes on $\widetilde{S}_{K-1}$. Therefore on $\widetilde{S}_{K-1}$ the vanishing of  $\Omega_K$ implies $P_{1,K}=\dots =P_{n-1,K}=0$. All other summands involves products of these functions as coefficients. Thus $\Omega_K$ vanishes if and only if $Z\in \widetilde{S}_{K-1}$ and $P_{1,K}=\dots =P_{n-1,K}=0$. Inspecting (\ref{e:even}) we see that at these points $P_{j,K}=P_{j,K-1}=0$ for $1\le j \le n-1$ and $P_{n,K}=P_{n,K-1}$. By $(V_{K-1})$ we have $P_{n,K}=P_{n,K-1}=1$. 

Going back one step to $K-1$ we have  \begin{equation*}
\begin{aligned}
P_{n,K-1}&=P_{n,K-2} \\
P_{n-1,K-1}&=P_{n-1,K-2}- z_{n(n-1),K-1}P_{n,K-1} \\
& \ \, \vdots \\
P_{1,K-1}&=P_{1,K-2} - \sum_{j=2}^{n} z_{j1,K-1}P_{j,K-1}
\end{aligned}
\end{equation*}
and since $P_{j,K-2}=0$ for $1 \le j \le n-1$ at points in $\widetilde{S}_{K-2}\supset \widetilde{S}_{K-1}$ we see that we must have  $$P_{j,K}=P_{j,K-1}=-z_{nj,K-1}=0$$ for $1\le j \le n-1$ at the points we are considering. This implies that the mapping is submersive exactly at points outside $S_K$. 

Finally we need justify the induction assumptions $(I_K)$, $(II_K)$, and $(III_K)$. 
We have already established $(III_K)$. We also see that  \begin{equation*}
\begin{aligned}
&dP_{1,K}\wedge \dots \wedge dP_{n-1,K}=\\ &=dP_{1,K-1}\wedge \dots \wedge dP_{n-1,K-1} + \text{ terms involving } dz_{jl,K}
\end{aligned}
\end{equation*}
and the first term is non-vanishing on $\widetilde{S}_{K-1}$ by $(IV_{K-1})$ (and thus on $S_K$). This establishes $(II_K)$. By (\ref{inductiveassumption}) we know that $dP_{n,K-1}$ vanishes identically on $\widetilde{S}_{K-1}$ and therefore on $S_K$ and this is $(I_K)$.
This concludes our induction step and the lemma follows.
\end{proof}

\begin{Rem}\label{r:smoothness}
Let us make some observations about the sets $S_K$ and how they are situated in relation to the fibers of the mapping. First when $K$ is even one sees that the image of $S_K$ using the map $\Phi_K\colon \left(\mathbb{C}^{n(n-1)/2} \right)^K \to \mathbb{C}^n\setminus \{ 0 \}$ is $(0,\dots, 0 , 1)$ so the points in $S_K$ are all contained in $\Phi^{-1}(\{(0,\dots, 0,1\})$. When $K$ is odd the image of $S_K$ is $\{(z_1,\dots,z_n)\in \mathbb{C}^n\setminus \{ 0 \}; z_n=1\}$.
\end{Rem}

\section{Proof of Lemma \ref{l:spray}}
\label{technical2}

\begin{Def}
We say that a polynomial $p(x_1,\dots,x_n)\in \mathbb{C}[\mathbb{C}^n]$ is no more than linear in $x_k$ if there exists two polynomials $\widetilde{p}, \widetilde{q}\in \mathbb{C}[\mathbb{C}^n]$ both independent of $x_k$ such that $$ p = x_k\widetilde{p}+\widetilde{q}.$$ 
\end{Def}

We need the following two lemmata.

\begin{Lem}\label{l:integrable}
Let $p(x_1,\dots,x_n)\in \mathbb{C}[\mathbb{C}^n]$ be a polynomial which is no more than linear in each variable. Then the vector fields $$V_{ij,p}=\frac{\partial p}{\partial x_i}\frac{\partial}{\partial x_j}- \frac{\partial p}{\partial x_j}\frac{\partial}{\partial x_i} $$ for $1\le i < j \le n$ are globally integrable on $\mathbb{C}^n$.
\end{Lem}
\begin{proof}
Note that $\partial p/\partial x_i$ is no more than linear in $x_j$ and independent of $x_i$ since $p$ is no more than linear in each variable separately. Hence the vector field $\partial p/\partial x_i (\partial / \partial x_j)$ is globally integrable and independent of $x_i$. Similarly the vector field $\partial p/\partial x_j (\partial / \partial x_i)$ is globally integrable and independent of $x_j$. Therefore the vector fields $V_{ij,p}$ are globally integrable.
\end{proof}
\begin{Lem}\label{l:span}
Let $p(x_1,\dots,x_n)\in \mathbb{C}[\mathbb{C}^n]$ and $$F_p(c)=\left\{X=(x_1,\dots,x_n)\in \mathbb{C}^n; p(X)=c \right\} $$ be the fiber of $p$ over the value $c$.  Then the vector fields $$V_{ij,p}=\frac{\partial p}{\partial x_i}\frac{\partial}{\partial x_j}- \frac{\partial p}{\partial x_j}\frac{\partial}{\partial x_i}, \ 1\le i < j \le n,$$ span the tangent space of $F_p(c)$ at all smooth points $X\in F_p(c)$ (i.e. those points where $dp$ does not vanish.)
\end{Lem}
\begin{proof}
We have $$V_{ij,p}(p-c)= \frac{\partial p}{\partial x_i}\frac{\partial p}{\partial x_j}- \frac{\partial p}{\partial x_j}\frac{\partial p}{\partial x_i}\equiv 0$$ so the vector fields are tangential to $F_p(c)$. We need to show that $$\dim  \text{span} \left ( V_{ij,p} ; 1\le i < j \le n \right )=n-1.$$ But this is obvious since at points where $dp\neq 0$ one component, say $\partial p / \partial x_n$, is non-zero and then $$ \dim \text{span} \left( V_{in,p} ; 1\le i \le n-1 \right ) = n-1.$$
\end{proof}

\begin{proof}[Proof of Lemma \ref{l:spray}]
In order to construct a spray we will produce globally integrable vector fields that span the tangent spaces of the fibers of $\Phi_K$. These fibers are given by $n$ polynomial equations in $Kn(n-1)/2$ variables. It is difficult to produce globally integrable vector fields that leave these polynomials invariant. The main goal of our proof will be to reduce, on each stratum individually, these polynomial equations to essentially a single polynomial equation. This polynomial equation will be no more than linear in each variable and therefore Lemma \ref{l:integrable} and Lemma \ref{l:span} will provide us with the desired integrable fields.. 

Recall that we have the relations 
\begin{equation}\label{e:evencase}
\begin{aligned}
P_{1,K}&=P_{1,K-1} \\
P_{2,K}&=P_{2,K-1} - z_{12,K}P_{1,K} \\
& \ \, \vdots \\
P_{k,K}&=P_{k,K-1} - \sum_{j=1}^{k-1} z_{jk,K}P_{j,K} \\
& \ \, \vdots \\
P_{n,K}&=P_{n,K-1} - \sum_{j=1}^{n-1} z_{jn,K}P_{j,K}
\end{aligned}
\end{equation}
when $K$ is even and 
\begin{equation}\label{e:oddcase}
\begin{aligned}
P_{n,K}&=P_{n,K-1} \\
P_{n-1,K}&=P_{n-1,K-1}- z_{n(n-1),K}P_{n,K} \\
& \ \, \vdots \\
P_{k,K}&=P_{k,K-1} - \sum_{j=k+1}^{n} z_{jk,K}P_{j,K} \\
& \ \, \vdots \\
P_{1,K}&=P_{1,K-1} - \sum_{j=2}^{n} z_{j1,K}P_{j,K}
\end{aligned}
\end{equation} 
when $K$ is odd. 

Lets make the following 

\medskip
{\bf Observation ($\star$):}

From (\ref{e:evencase}) and (\ref{e:oddcase}) one easily deduce by induction that the map $\Phi_K=(P_{1,K},\dots,P_{n,K})$ has polynomial entries that are no more than linear in each variable. Using (\ref{e:evencase}) one sees that $P_{1,K}$ is independent of $Z_K$, $P_{2,K}$ depends only on $Z_1,\dots, Z_{K-1}, z_{12,K}$, and in general $P_{k,K}$ depends only on $Z_1,\dots, Z_{K-1}$ and $z_{ij,K}$ for $1\le i < j \le k$ when $K$ is even. When $K$ is odd one conclude, using (\ref{e:oddcase}), that $P_{k,K}$ depends only on $Z_1,\dots, Z_{K-1}$ and $z_{ij,K}$ for $k\le j < i \le n$.

\medskip

We will begin by considering the case $K$ even: 

Here we will stratify $\mathbb{C}^n\setminus \{ 0 \}$ as 
\begin{itemize}
\item{$V_n=\mathbb{C}^n\setminus \{ 0 \}$,} 
\item{$V_{n-k}=\{ (z_1,\dots, z_n)\in \mathbb{C}^n\setminus \{ 0 \};z_1=\dots = z_k=0\} \text{ when } 1\le k \le n-1$,  and} 
\item{$V_0=\emptyset$.}
\end{itemize} 
First consider a fiber over a point  $a=(a_1,\dots,a_n)\in V_n\setminus V_{n-1}$. Here we have \begin{equation*}
\begin{aligned}
P_{1,K}&= a_1\neq 0 \\
P_{2,K}&=P_{2,K-1} - z_{12,K}P_{1,K} = a_2 \\
& \ \, \vdots \\
P_{n,K}&=P_{n,K-1} - \sum_{j=1}^{n-1} z_{jn,K}P_{j,K} = a_n.
\end{aligned}
\end{equation*}
Put $Z_K=(Z'_K,Z''_K)$ where $Z''_K=(z_{12,K},\dots,z_{1n,K})$ and $Z'_K$ consists of the other variables in $Z_K$.  
We see that the fiber $\Phi_K^{-1}(a)$ is biholomorphic to (it is a graph over) 

$$B_n(a)=\{Z=(Z_1,\dots,Z'_K)\in (\mathbb{C}^{n(n-1)/2})^{K-1}\timesÊ\mathbb{C}^{(n-1)(n-2)/2};  P_{1,K}(Z)=a_1\}$$ (remember that $P_{1,K}$ is independent of $Z_K$) since 
\begin{equation*}
\begin{aligned}
z_{12,K}&= \frac{P_{2,K-1}-a_2}{a_1} \\
& \ \, \vdots \\
z_{1n,K}&=\frac{P_{n,K-1}-a_n - \sum_{j=2}^{n-1} z_{jn,K}a_j}{a_1}.
\end{aligned}
\end{equation*}
This also shows that the fibration $\Phi_K\colon \Phi_K^{-1}(V_n\setminus V_{n-1}) \to V_n\setminus V_{n-1}$ is biholomorphic to the fibration $$\xymatrix{ \{ (Z_1,\dots,Z'_K,a)\in \mathbb{C}^{M}\times (V_n\setminus V_{n-1}); P_{1,K}(Z)=a_1 \} \ar[d]^{(Z_1,\dots, Z'_K,a) \mapsto a} \\   V_n\setminus V_{n-1}}$$ 
where $\mathbb{C}^M= (\mathbb{C}^{n(n-1)/2})^{K-1}\timesÊ\mathbb{C}^{(n-1)(n-2)/2}$. The vector fields $V_{ij,P_{1,K}}$ where $i,j$ run through all pairs of variables in $\mathbb{C}^M$ are globally integrable and span the tangent space of each individual fiber by Lemma \ref{l:integrable} and \ref{l:span} (since the fibers of $\Phi_K$ over $V_n\setminus V_{n-1}$ are smooth and biholomorphic to $\{P_{1,K}(Z)=a_1\}$, see Remark \ref{r:smoothness}). This gives us the vector fields needed to conclude that the restricted submersion over $V_n\setminus V_{n-1}$ admits a spray. 

Next lets study the fiber over a point $a=(0,a_2,\dots, a_n) \in V_{n-1}\setminus V_{n-2}$. Here the relations for the fiber are 
\begin{equation*}
\begin{aligned}
P_{1,K}&=P_{1,K-1}=  0 \\
P_{2,K}&=P_{2,K-1} - z_{12,K}P_{1,K} = a_2 \neq 0 \\
& \ \, \vdots \\
P_{n,K}&=P_{n,K-1} - \sum_{j=1}^{n-1} z_{jn,K}P_{j,K} = a_n.
\end{aligned}
\end{equation*}
Since $P_{1,K}=0$ the system is equivalent to  
\begin{equation*}
\begin{aligned}
P_{1,K}&=P_{1,K-1}=  0 \\
P_{2,K}&=P_{2,K-1}= a_2 \neq 0 \\
& \ \, \vdots \\
P_{n,K}&=P_{n,K-1} - \sum_{j=2}^{n-1} z_{jn,K}P_{j,K} = a_n.
\end{aligned}
\end{equation*}
and $z_{12,K},\dots,z_{1n,K}$ are free variables. As in the case above we can (using $a_2\neq 0$) solve the last $n-2$ equations for the variables $z_{23,K}, \dots , z_{2n,K}$, namely 
\begin{equation*}
\begin{aligned}
z_{23,K}&= \frac{P_{3,K-1}-a_3}{a_2} \\
& \ \, \vdots \\
z_{2n,K}&=\frac{P_{n,K-1}-a_n - \sum_{j=3}^{n-1} z_{jn,K}a_j}{a_2}.
\end{aligned}
\end{equation*}
Put $Z_K=(Z'_K,Z''_K)$ where $Z''_K=(z_{12,K},\dots,z_{1n,K},z_{23,K},\dots,z_{2n,K})$ and $Z'_K$ consists of the other variables in $Z_K$. We see that the fiber $\Phi_K^{-1}(a)$ is biholomorphic to
\begin{equation*}
\begin{aligned}
B_{n-1}(a)=\{Z=(Z_1,\dots,Z'_K)\in (\mathbb{C}^{n(n-1)/2})^{K-1}\timesÊ\mathbb{C}^{(n-2)(n-3)/2}; \\ P_{1,K}(Z)=0, P_{2,K}(Z)=a_2\}\times  \mathbb{C}^{n-1}_{(z_{12,K},\dots,z_{1n,K})}.
\end{aligned}
\end{equation*}
This system of 2 equations can be reduced to one equation by going back one step and using the last equation of (\ref{e:oddcase}) which says 
\begin{equation*}
P_{1,K-1}=\left(P_{1,K-2} - \sum_{j=3}^nz_{j1,K-1}P_{j,K-1} \right) - z_{21,K-1}P_{2,K-1}=0.
\end{equation*}
It allows us to solve for $z_{21,K-1}$: 
\begin{equation*}
z_{21,K-1}= \frac{\left(P_{1,K-2} - \sum_{j=3}^nz_{j1,K-1}P_{j,K-1} \right)}{a_2}
\end{equation*}
From Observation ($\star$) we see that $$ \left(P_{1,K-2} - \sum_{j=3}^nz_{j1,K-1}P_{j,K-1} \right) $$ does not depend on $z_{21,K-1}$. Putting 
\begin{equation*}
\mathbb{C}^M= (\mathbb{C}^{n(n-1)/2})^{K-2}\times \mathbb{C}^{n(n-1)/2-1} \timesÊ\mathbb{C}^{(n-2)(n-3)/2}
\end{equation*}
and
\begin{equation*}
X=(Z_1,\dots,Z_{K-2},\dots, \widehat{z}_{21,K-1},\dots,Z'_{K})\in \mathbb{C}^M
\end{equation*}
we have shown that the fibration $\Phi_K\colon \Phi_K^{-1}(V_{n-1}\setminus V_{n-2}) \to V_{n-1}\setminus V_{n-2}$ is biholomorphic to the fibration $$\xymatrix{ \{ (X,a)\in \mathbb{C}^{M}\times (V_{n-1}\setminus V_{n-2}); P_{2,K-1}(X)=a_2 \}\times \mathbb{C}^{n-1}_{(z_{12,K},\dots, z_{1n,K})} \ar[d]^{(X,z_{12,K},\dots,z_{1n,K},a) \mapsto a} \\   V_{n-1}\setminus V_{n-2}}$$ 
and the spray is constructed as above.

For general $k$ we proceed analogously. Using $a_1=\dots =a_{k-1}=0$ we get free variables $z_{ij,K}$ for $1\le i \le k-1$ and $i<j \le n$ and the first $k$ equations become $$P_{1,K-1}=\dots =P_{k-1,K-1}=0$$ and $$P_{k,K-1}=a_k.$$ Using $a_k\neq 0$ we solve the last $n-k$ equations for the variables $z_{k(k+1),K},\dots, z_{kn,K}$. Then we go one step back and use the last $k-1$ equations of (\ref{e:oddcase}) to rewrite our first $k-1$ equations. This allow us to solve for the variables $z_{k1,K-1},z_{k2,K-1},\dots, z_{k(k-1),K-1}$. Also the variables $z_{ij,K-1}$, for $1\le j \le k-2$ and $j< i \le k-1$, become free variables. This shows that the fibration  $\Phi_K\colon \Phi_K^{-1}(V_{n-k+1}\setminus V_{n-k}) \to V_{n-k+1}\setminus V_{n-k}$ is biholomorphic to the fibration $$\xymatrix{ \{ (X,a)\in \mathbb{C}^{M}\times (V_{n-k+1}\setminus V_{n-k}); P_{k,K-1}(X)=a_k \}\times \mathbb{C}^{N}_w \ar[d]^{(X,w,a) \mapsto a} \\   V_{n-k+1}\setminus V_{n-k}}$$ 
for appropriate $N$. The globally integrable vector fields $V_{ij,P_{k,K}}$ where $i,j$ run through all pairs of variables in $\mathbb{C}^M$ together the fields $\partial/\partial w_l$ where $l$ runs over all variables in $\mathbb{C}^N_w$ give the spray on the (smooth part, i.e., when $S_K$ is removed, of the) fibers over $ V_{n-k+1}\setminus V_{n-k}$. Remembering Remark (\ref{r:smoothness}) we have smooth fibers over all strata except for the very last stratum where we have a non-smooth fiber over $(0,\dots,0,1)$.

When $K$ is odd the method is basically the same. Only here the stratification is $V_n=\mathbb{C}^n\setminus \{ 0 \}$, $V_{n-k}=\{(z_1,\dots,z_n)\in \mathbb{C}^n\setminus \{ 0 \}; z_{n-k+1}=\dots=z_n=0\}$ when $1 \le k \le n-1$, and $V_0=\emptyset$. On $V_{n-k+1}\setminus V_{n-k}$ one show that the fibers are biholomorphic to $\{P_{n-k+1,K-1}=a_{n-k+1}\}$ times free variables. Note that the singular fibers are contained in the first stratum $V_n\setminus V_{n-1}$, see Remark \ref{r:smoothness}. 
\end{proof}

\begin{Rem}
We do not know whether there is a possibly finer stratification so that the restricted submersions are locally trivial fiber bundles. In some cases we see from the explicit form of the polynomials that we have local triviality. We have not been able to decide this in all cases.
\end{Rem}

\section{On the number of factors}

A natural question to ask is how the number of factors needed in the factorization depends
on the space $X$ and the map $f$. In the algebraic setting there is no such uniform bound as proved by \textsc{van der Kallen} in \cite{vanderKallenSL3BWL}. However in the holomorphic setting (exactly as in the topological setting) it is easy to see that there is an upper
bound depending only on the dimension  of the space $X$ ($= m$) and the size of the matrix
($=n$). 

It follows from \textsc{Vaserstein's} result  (theorem \ref{Vaserstein}) that there exists a 
uniform bound $K$ depending on the dimension of the space $X$ ($= m$) and the size of the matrix
$=n$ such that the fibration 
$$\xymatrix{ p^{\star}Y \ar[d]^{p^{\star}\Phi_{K}} \\   X}$$ 
from the proof of proposition \ref{p:mainprop} has a topological section and hence a holomorphic section. Going through the induction over the size of the matrix as in the proof of Theorem \ref{t:mainthmrestate}
we conclude that there is a uniform bound even in the holomorphic case. 

Another way to prove the existence of such a uniform bound is the following. Suppose it would not exist, i.e., for all natural numbers $i$ there are Stein spaces $X_i$ of dimension $m$ and holomorphic maps $f_i \colon X_i\to \mbox{SL}_n(\mathbb{C})$ such that $f_i$ does not factor over a product of less than $i$ unipotent matrices. Set $X = \cup_{i=1}^\infty X_i$ the disjoint union of the spaces $X_i$ and
$F\colon X\to \mbox{SL}_n(\mathbb{C})$ the map that is equal to $f_i$ on $X_i$. By our main result
$F$ factors over a finite number of unipotent matrices. Consequently all $f_i$ factor over the same number of unipotent matrices which contradicts the assumption on $f_i$.

Thus we proved 

\begin{Thm}
There is a natural number $K$ such that for
any   reduced Stein space $X$ of dimension $m$ and any null-homotopic holomorphic mapping $f\colon X\to \mbox{SL}_n(\mathbb{C})$  there exist holomorphic mappings $G_1,\dots, G_{K}\colon X\to \mathbb{C}^{n(n-1)/2}$ such that $$f(x)=M_{1}(G_1(x))\dots M_{K}(G_{K}(x)).$$
\end{Thm}

Let us denote by   $K_{\mathcal C}(m,n)$ the number of matrices needed to factorize any null-homotopic map from a Stein space of dimension $m$ into $\mbox{SL}_n(\mathbb{C})$ by continuous triangular matrices and the number needed in the holomorphic case by $K_{\mathcal O}(m,n)$. We do not know these numbers. We know that the Cohn example can be factored as 4 matrices with continuous entries but if one wants to factor it using matrices with holomorphic entries one needs 5 matrices.  It is natural to ask the following question:

\begin{Prob}
How are the numbers $K_{\mathcal C}(m,n)$ and $K_{\mathcal O}(m,n)$ exactly related? Obviously  $K_{\mathcal C}(m,n) \le K_{\mathcal O}(m,n)$.
\end{Prob}

Examining our proof in the case $n=2$ one easily deduces the estimate  $K_{\mathcal O}(m,2) \le  K_{\mathcal C}(m,2) + 4$. At least for the case $n=2$ we believe the answer to the above  question can be found.


\end{document}